\documentclass[12pt]{amsart}
\usepackage{amssymb}
\usepackage{amsthm}
\usepackage{amsmath}
\usepackage{graphicx}
\usepackage{float}
\usepackage{enumerate}
\usepackage{color}
\usepackage{longtable}
\restylefloat{table}

\newtheorem{theorem}{Theorem}[section]

\newtheorem{definition}[theorem]{Definition}
\newtheorem{lemma}[theorem]{Lemma}
\newtheorem{corollary}[theorem]{Corollary}
\newtheorem{remark}[theorem]{Remark}

\def\F{\mathbb{F}}

\def\V{\mathcal{V}}
\def\Ll{\mathcal{L}}
\def\B{\mathcal{B}}
\def\R{\mathcal{R}}
\def\K{\mathcal{K}}
\def\Fq{\mathbb{F}_q}
\def\I{\mathcal{I}}

\DeclareMathOperator{\PGammaL}{P{\Gamma}L}
\DeclareMathOperator{\Aut}{Aut}

\newcommand{\comments}[1]{}
\newcommand{\gs}[3]{\genfrac{[}{]}{0pt}{0}{#1}{#2}_{#3}}

\begin{document}

\title[The existence of 3-designs over finite fields]{Necessary conditions for the existence of 3-designs over finite fields with nontrivial automorphism groups}

\author{Maarten De Boeck}
\address{UGent, Department of Mathematics, Krijgslaan 281-S22, 9000 Gent, Flanders, Belgium}
\email{mdeboeck@cage.ugent.be} 
\author{Anamari Naki\'c}
\address{University of Zagreb, Faculty of Electrical Engineering and Computing, Unska 3, 10000 Zagreb, Croatia}
\email{anamari.nakic@fer.hr}



\date{September, 2015}

\begin{abstract}
 A $q$-design with parameters $t$-$(v,k,\lambda_{t})_q$ is a pair $(\V, \B)$ of the vector space $\V=\F^{v}_{q}$ and a collection $\B$ of $k$-dimensional subspaces of $\V$, such that each $t$-dimensional subspace of $\V$ is contained in precisely $\lambda_t$ members of $\B$. In this paper we give new general necessary conditions on the existence of designs over finite fields with parameters $3$-$(v, k , \lambda_3)_q$ with a prescribed automorphism group. These necessary conditions are based on a tactical decomposition of such a design over a finite field and are given in the form of equations for the coefficients of tactical decomposition matrices. In particular, they represent necessary conditions on the existence of $q$-analogues of Steiner systems admitting a prescribed automorphism group.
\end{abstract}

\maketitle

\section{Introduction}

\begin{definition}
  A \emph{$q$-design} with parameters $t$-$(v,k,\lambda_{t})_q$, or shorter a $t$-$(v,k,\lambda_t)_q$ design, with $v > k > 1, k \geq t \geq 1$, $\lambda_t \geq 1$, is a pair $(\V, \B)$ of the vector space $\V=\F^{v}_{q}$ and a collection $\B$ of $k$-dimensional subspaces of $\V$ (called blocks), such that each $t$-dimensional subspace of $\V$ is contained in precisely $\lambda_t$ blocks.
\end{definition}

If $\B$ consists of all the $k$-dimensional subspaces of $\V$, the design $(\V, \B)$ is called \emph{trivial}. Throughout this paper we shall call the $1$-spaces of $\V$ \emph{points} and often, when convenient, identify $\V$ with its set of points.

$q$-designs recently attract considerable attention. They represent a generalisation of classical designs in terms of vector spaces, a so-called $q$-analogue, and were first introduced in the 1970's, see \cite{CI74, CII74, D76}. The recent application of these $q$-designs in error-correction in randomized network coding, made this topic more interesting than ever~\cite{E13,EV11}. 

In random network coding, information is transmitted through a network whose topology can vary. A classical example is a wireless network where users come and go. For more details see \cite{ACLY00, HMK, KK08, L13, YLCZ06}. It was showed in~\cite{KK08} that subspace codes are well-suited for transmission in networks. A {\em subspace code} is a set of $k$-dimensional vector subspaces of the vector space $\F^{v}_q$. This new insight led to many new interesting problems in coding theory, in Galois geometries and in design theory (for a general overview see~\cite{E13}). E.g., it has been noted (see~\cite{ACLY00, EV11}) that $q$-analogues of Steiner systems, briefly called {\em $q$-Steiner systems}, are optimal subspace codes. These $q$-Steiner systems are $q$-designs with $\lambda_t=1$. The existence problem of $q$-Steiner systems is still open. It is known that a $1$-$(v, k, 1)_q$ Steiner system exists if and only if $k$ divides $v$, and in this case it is called a spread. The first $q$-Steiner systems which are not spreads were constructed recently in~\cite{BEOVW15}: using algorithms based on automorphism groups a $2$-$(13,3,1)_2$ design was computationally constructed. Currently, one of the most interesting open problems is the existence of the $q$-analogue of the Fano plane, a $2$-$(7,3,1)_q$ design (see~\cite{BKN15, EV11, KP14, M99, MMY95, ST87}).

Also for general $q$-designs the existence problem is mainly unsolved. Examples of designs over finite fields with $t \ge 2$ constructed so far are mostly $q$-designs with $t=2$ \cite{BEOVW15, BKL07, MMY95, S90, S92, ST87}.

Therefore, an interesting research direction is the study of $q$-designs with $t > 2$. Only a few examples of non-trivial $3$-designs over finite fields were constructed: a $3$-$(8,4,11)_2$ design and a $3$-$(8,4,20)_2$ design in \cite{BKL07}, and a $3$-$(8,4,15)_2$ design in \cite{B13}. They were all constructed computationally with algorithms using their automorphism group. 

One of the open problems posed in~\cite{E13} is finding new necessary conditions on the existence of $q$-Steiner systems. In this paper we give new general necessary conditions on the existence of $3$-$(v, k , \lambda_3)_q$ designs with a prescribed automorphism group. These necessary conditions are based on a tactical decomposition of such a $q$-design and are given in the form of equations for the coefficients of tactical decomposition matrices. In \cite{nak1} tactical decompositions of designs over finite fields with $t = 2$ were studied. It was shown there that coefficients of tactical decomposition matrices satisfy an equation system analogue to the one known for classical block designs (see Section~\ref{sec:2-designs}). The main result is given in  Theorem~\ref{maintheorem}. Further in Section~\ref{sec:3-designs} we show that for $t = 3$, the system of equations for $q$-designs is not equivalent to the one for classical 3-designs. Crucial in the main theorem are the values $\Lambda_{lrs}$; additional results about these $\Lambda$-values are presented in Section~\ref{sec:parameters}. Finally, in Section~\ref{examples} we present some applications of the main theorem to known $q$-designs.

\section{Preliminaries}

The number of $r$-dimensional subspaces of the vector space $\F^{v}_{q}$ is
\[
\gs{v}{r}{q} = \frac{(q^v-1)(q^{v-1}-1)\dots(q^{v-r+1}-1)}{(q^r-1)(q^{r-1}-1)\dots(q-1)}.
\]
The number of $r$-dimensional subspaces of $\V$ containing a fixed $s$-dimensional subspace, $s\le r$, equals $\gs{v-s}{r-s}{q}$. For every two subspaces $U$ and $W$ of a vector space, \emph{the dimension formula} is valid:
\[
\dim\,\langle U, W \rangle = \dim \, U + \dim \, W - \dim\,(U \cap W).
\]

If $(\V, \B)$ is a $t$-$(v,k,\lambda_t)_q$ design, then it is also a $q$-design with parameters $s$-$(v,k,\lambda_s)_q$, $0\le s\le t$, with
\begin{align}\label{tdesignissdesign}
  \lambda_s = \lambda_t \frac{\gs{v-s}{t-s}{q}}{\gs{k-s}{t-s}{q}}.
\end{align}
The number of blocks in $\B$ equals
\[
  |\B|=\lambda_{0}=\lambda_t \frac{\gs{v}{t}{q}}{\gs{k}{t}{q} }.
\]

$q$-designs are closely related to classical designs, as they are the $q$-analogues of the classical designs. A $t$-$(v,k,\lambda_t)$ \emph{design} is a finite incidence structure $ ({\mathcal P}, {\mathcal B} ) $, where $\mathcal{P}$ is a set of $v$ elements called \emph{points}, and $\mathcal{B}$ is a multiset of nonempty $k$-subsets of $\mathcal{P}$ called \emph{blocks} such that every set of $t$ distinct points is contained in precisely $\lambda_t$ blocks. Every $2$-$(v, k, \lambda_2)_q$ design gives rise to a classical design with parameters $2$-$\left(\gs{v}{1}{q},\gs{k}{1}{q},\lambda_2\right)$ by identifying the points of $\V$ with the points of the design and each block in $\B$ with the set of points it contains. The inverse statement is not valid. E.g. there are classical designs with parameters $2$-$(15,7,3)$ which cannot be constructed from the unique $2$-$(4,3,3)_2$ design~\cite{MR07}.

An \emph{automorphism} of the $q$-design $(\V, \B)$ is a map $g\in\PGammaL(\V)$ such that $\B^{g} = \B$. The set $\Aut(\V, \B)$ of all automorphisms of $(\V, \B)$ is a subgroup of $\PGammaL(\V)$, called the \emph{full automorphism group} of $(\V, \B)$. We say that $(\V, \B)$ admits the finite group $G$, or equivalently that $G$ is an automorphism group of $(\V, \B)$, if there is a subgroup of $\Aut(\V, \B)$ isomorphic to $G$.

\section{$q$-designs with a tactical decomposition}\label{sec:2-designs}

In this section we address the definition and known results concerning automorphism groups and tactical decompositions of $q$-designs. The idea of considering tactical decompositions of classical block designs was first introduced by Dembowski~\cite{D68}. Tactical decomposition has been crucial for the construction of many classical $2$-designs~\cite{JT85, MR07}. In \cite{nak1} tactical decompositions of $q$-designs with $t = 2$ were studied. 

\begin{definition}
Let $(\V, \B)$ be a $q$-design. A \emph{decomposition} of $(\V, \B)$ consists of two partitions
\[
  \V = \V_1 \sqcup \cdots \sqcup \V_m, \quad \B = \B_1 \sqcup \cdots \sqcup \B_n.
\]
We say that a decomposition is \emph{tactical} if there exist nonnegative integers $\rho_{ij}$, $\kappa_{ij}$, $i=1,\dots,m$, $j=1,\dots,n$, such that
  \begin{enumerate}
    \item every point of $\V_i$ is contained in $\rho_{ij}$ blocks of $\B_j$, 
    \item each block of $\B_j$ contains $\kappa_{ij}$ points of $\V_i$.
  \end{enumerate}
The matrices $\R = [\rho_{ij}]$ and $\K = [\kappa_{ij}]$ are called the \emph{tactical decomposition matrices}.
\end{definition}

There are two trivial examples of a tactical decomposition of a $q$-design. The first example is obtained by putting $n = m = 1$, and the second by partitioning sets $\V$ and $\B$ into singletons. A nontrivial tactical decomposition can be obtained by the action of an automorphism group $G \le\Aut(\V, \B)$ on a design.

\begin{theorem}\label{groupinduced}
	Let $G$ be an automorphism group of a design $(\V, \B)$ over a finite field. Then the orbits of the set of points $\V$ and the orbits of the set of blocks $\B$ form a tactical decomposition.
\end{theorem}
\begin{proof}
	Let $\V_i$ be a point orbit and $\B_j$ be an orbit of $\B$ under the action of $G$. The statement follows immediately from the observation that $P\in\V_{i}$ is contained in $B\in\B_{j}$ if and only if $P^{g}\in\V_{i}$ is contained in $B^{g}\in\B_{j}$ for any $g\in G$.
\end{proof}

A tactical decomposition that arises from a group action as in Theorem \ref{groupinduced} is called \emph{group-induced}. If the group $G$ is specified, we call it \emph{$G$-induced}.
\par The following result is valid for all designs over finite fields. 

\begin{lemma}[{\cite[Section 2]{nak1}}]\label{firstequations}
	Let $(\V, \B)$ be a design with parameters $t$-$(v,k,\lambda_{t})_q$ that admits a tactical decomposition
	\[
	\V = \V_1 \sqcup \cdots \sqcup \V_m,\quad \B = \B_1 \sqcup \cdots \sqcup \B_n,
	\]
	with tactical decomposition matrices $\R = [\rho_{ij}]$ and $\K = [\kappa_{ij}]$. Then,
	\begin{align*}
	\sum^{m}_{i=1}\kappa_{ij}=\gs{k}{1}{q},\qquad
	\sum^{n}_{j=1}\rho_{ij}=\lambda_{1},
	\end{align*}
	and for all $i=1,\dots,m$ and $j=1,\dots,n$,
	\[
	|\V_{i}|\cdot\rho_{ij}=|\B_{j}|\cdot\kappa_{ij}.
	\]
\end{lemma}

It was shown in \cite{nak1} that coefficients of a tactical decomposition matrix of a $q$-design with $t=2$ satisfy an equation system analogous to the one known for classical block designs with $t=2$.

\begin{theorem}\label{secondequations} \cite{nak1} 
	Assume $(\V, \B)$ is a $2$-$(v,k,\lambda_2)_q$ design over finite field or a classical $2$-$(v,k,\lambda_2)$ design with a tactical decomposition 
	$$
	\V = \V_1 \sqcup \cdots \sqcup \V_m,\,\,\,
	{\mathcal B} = {\mathcal B}_1 \sqcup \cdots \sqcup {\mathcal B}_n.
	$$
	Let $[\rho_{ij}]$ and $[\kappa_{ij}]$ be the associated tactical decomposition matrices.
	Then
	\begin{align}\label{eq2}
	\sum_{j = 1}^{n} \rho_{l j} \kappa_{r j} =
	\begin{cases}
	  \lambda_1 + \lambda_2 \cdot (|\V_r|-1) & l = r\\
	  \lambda_2 \cdot |\V_r| & l \neq r.
	\end{cases}
	\end{align}
\end{theorem}

Note that the right-hand side of (\ref{eq2}) only contains parameters of the design that can easily be computed.

\section{$3$-designs over finite fields with nontrivial automorphism groups}\label{sec:3-designs}

We now investigate designs over finite fields with $t=3$ having a nontrivial tactical decomposition. 
We introduce the following notation. 

\begin{definition}\label{bigM}
  Let $\V = \V_1 \sqcup \cdots \sqcup \V_m$ be a partition of the vector space $\V$. For a point $P \in \V_l$, we define the parameters 
  \[
    \Lambda_{rs}(P) = \left|\{(R, S) \in \V_{r}\times\V_s\, : \, \dim\langle P, R, S \rangle = 2\}\right|.
  \]
\end{definition}

\begin{lemma}\label{dclemma}
   Let $(\V, \B)$ be a $3$-$(v,k,\lambda_{3})_q$ design that admits a tactical decomposition
   \[
     \V = \V_1 \sqcup \cdots \sqcup \V_m,\quad \B = \B_1 \sqcup \cdots \sqcup \B_n,
   \]
   with tactical decomposition matrices $\R = [\rho_{ij}]$ and $\K = [\kappa_{ij}]$. Let $P$ be a point in $\V_{l}$.  Then,
   \[
     \sum_{j=1}^{n} \rho_{lj}\kappa_{rj}\kappa_{sj} =
     \begin{cases}
       \lambda_{1}+\Lambda_{ll}(P)\cdot\lambda_{2}+\left(|\V_{l}|^{2}-\Lambda_{ll}(P)-1\right)\cdot\lambda_{3} & l=r=s\\
       \Lambda_{rs}(P)\cdot\lambda_{2}+\left(|\V_{r}|\cdot|\V_{s}|- \Lambda_{rs}(P)\right)\cdot\lambda_{3} & \text{otherwise.}
     \end{cases}
   \]
\end{lemma}
\begin{proof}
  Double-counting of the set of triples
  \[
    \{(R,S,B) \in \V_{r} \times \V_{s} \times \B\, : \,  P, R, S \le B \}
  \]
  yields
  \begin{align}\label{doublecounting}
    \sum_{j=1}^{n} \rho_{lj}\kappa_{rj}\kappa_{sj} = \sum_{R \in \V_r} \sum_{S \in \V_{s}} | \I_{P} \cap \I_{R} \cap \I_{S}|,
  \end{align}
  with $\I_{Q}$ the subset of blocks of $\B$ that contain the point $Q$, for any point $Q$. Now, consider $R \in \V_r$ and $S \in \V_{s}$. It is immediate that
  \[
    \I_{P} \cap \I_{R} \cap \I_{S} = \{  B \in \B \, : \, \left\langle P, R, S \right\rangle\le B \}.
  \]
  It is clear that $1\leq\dim\left\langle P,R,S \right\rangle\leq3$ and so
  \[
    \left|\I_{P} \cap \I_{R} \cap \I_{S}\right|=\lambda_{\dim\left\langle P,R,S \right\rangle}.
  \]
  Hence, in order to find an expression for \eqref{doublecounting} it is sufficient to count the number of pairs $(R,S)\in\V_{r}\times\V_{s}$ such that $\dim\left\langle P,R,S \right\rangle=i$, for $i=1,2,3$. It is clear that $\dim\left\langle P,R,S \right\rangle=1$ if and only if $P=R=S$ and thus $l=r=s$. Consequently,
  \begin{align*}
    &\sum_{R \in \V_r} \sum_{S \in \V_{s}} | \I_{P} \cap \I_{R} \cap \I_{S}|\\&=
    \begin{cases}
      \lambda_{1}+\Lambda_{ll}(P)\cdot\lambda_{2}+\left(|\V_{l}|^{2}-\Lambda_{ll}(P)-1\right)\cdot\lambda_{3}, & l=r=s\\
      \Lambda_{rs}(P)\cdot\lambda_{2}+\left(|\V_{r}|\cdot|\V_{s}|-\Lambda_{rs}(P)\right)\cdot\lambda_{3}, & \text{otherwise,}
    \end{cases}
  \end{align*}
  from which the result follows, using \eqref{doublecounting}.
\end{proof}

\begin{corollary}\label{independent}
  Let $(\V, \B)$ be a $3$-$(v,k,\lambda_{3})_q$ design that admits a tactical decomposition
  \[
    \V = \V_1 \sqcup \cdots \sqcup \V_m,\quad \B = \B_1 \sqcup \cdots \sqcup \B_n.
  \]
  Consider $l,r,s\in\{1,\dots,m\}$. The values $\Lambda_{rs}(P)$ are independent of the choice of $P\in\V_{l}$.
\end{corollary}
\begin{proof}
  Let $\R = [\rho_{ij}]$ and $\K = [\kappa_{ij}]$ be the tactical decomposition matrices of this tactical decomposition. By Lemma \ref{dclemma} we know that
  \begin{align*}
    \sum_{j=1}^{n} \rho_{lj}\kappa_{rj}\kappa_{sj} &=
    \begin{cases}
      \Lambda_{ll}(P)\cdot(\lambda_{2}-\lambda_{3})+\lambda_{1}+\left(|\V_{l}|^{2}-1\right)\cdot\lambda_{3}, & l=r=s\\
      \Lambda_{rs}(P)\cdot(\lambda_{2}-\lambda_{3})+|\V_{r}|\cdot|\V_{s}|\cdot\lambda_{3}, & \text{otherwise.}
    \end{cases}
  \end{align*}
  for a point $P\in \V_{l}$. The left-hand side in this equation is clearly independent of the choice of $P$, hence also the right-hand side is independent of the choice of $P$. As $\lambda_{1},\lambda_{2},\lambda_{3},|\V_{r}|,|\V_{s}|$ are obviously $P$-independent and since $\lambda_{2}-\lambda_{3}\neq0$, necessarily also $\Lambda_{rs}(P)$ is independent of the choice of $P\in\V_{l}$.
\end{proof}

Following the result of Corollary \ref{independent}, we can define $\Lambda_{lrs}$ as $\Lambda_{rs}(P)$ for a point $P\in\V_{l}$. Using this notation, we state the main theorem of this section. 


\begin{theorem}\label{maintheorem}
	Let $(\V, \B)$ be a design over finite field with parameters $3$-$(v,k,\lambda_{3})_q$ that admits a tactical decomposition
	\[
	\V = \V_1 \sqcup \cdots \sqcup \V_m,\quad \B = \B_1 \sqcup \cdots \sqcup \B_n,
	\]
	with tactical decomposition matrices $\R = [\rho_{ij}]$ and $\K = [\kappa_{ij}]$. Then,
	\begin{align*}
	&\sum_{j=1}^{n} \rho_{lj}\kappa_{rj}\kappa_{sj} =
	\begin{cases}
	  \lambda_{1}+\Lambda_{lll}\cdot\lambda_{2}+\left(|\V_{l}|^{2}-\Lambda_{lll}-1\right)\cdot\lambda_{3} & l=r=s\\
	  \Lambda_{lrs}\cdot\lambda_{2}+\left(|\V_{r}|\cdot|\V_{s}|-\Lambda_{lrs}\right)\cdot\lambda_{3} & \text{otherwise}. \\
	\end{cases}
	\end{align*}
\end{theorem}

Note that a $q$-design with $t=3$ is also a $q$-design with $t=2$. Hence, Lemma \ref{firstequations} and Lemma \ref{secondequations} also present necessary conditions for the existence of $q$-designs with $t=3$, with a given tactical decomposition.

These parameters $\Lambda_{lrs}$ are not present in the discussion of classical $3$-designs as in \cite{KNP11, KNP13}, and form the main difference between the designs over finite fields and the classical designs with $t=3$. 
In order to compare we state here the analogue of Theorem \ref{maintheorem} for classical designs with $t = 3$.

\begin{theorem}\cite{KNP11, KNP13}
Let $(\V, \B)$ be a $3$-$(v,k,\lambda_{3})$ design that admits a tactical decomposition
\[
  \V = \V_1 \sqcup \cdots \sqcup \V_m,\quad \B = \B_1 \sqcup \cdots \sqcup \B_n,
\]
with tactical decomposition matrices $\R = [\rho_{ij}]$ and $\K = [\kappa_{ij}]$. Then,
	\begin{align*}
	&\sum_{j=1}^{n} \rho_{l j} \kappa_{r j} \kappa_{s j}\\ 
	&=\begin{cases}
	    \lambda_1 + 3\,(|\V_l|-1)\cdot \lambda_2 + (|\V_l|-1)\cdot(|\V_l|-2)\cdot \lambda_3 & l = r = s\\
	    \lambda_{2}\cdot|\V_{r}|+\lambda_3\cdot|\V_{r}|\cdot(|\V_{r}|-1) & l\neq r=s\\
	    |\V_r|\cdot |\V_s|\cdot \lambda_3+(\lambda_2-\lambda_3)\cdot(\delta_{lr}\cdot|\V_s| + \delta_{ls}\cdot|\V_r|)  & \mbox{otherwise}.
	  \end{cases} 
	\end{align*}
\end{theorem}

\section{Some results regarding $\Lambda_{lrs}$}\label{sec:parameters}

In this section we have a look at the parameter $\Lambda_{lrs}$ which we introduced in Definition \ref{bigM}. In order to use Theorem \ref{maintheorem} when computationally constructing a $q$-design, one needs to know these values $\Lambda_{lrs}$ for the given tactical decomposition. First we present three general results on these values $\Lambda_{lrs}$ and then we look at a specific case.

\begin{theorem}
  Let $(\V, \B)$ be a $3$-$(v,k,\lambda_{3})_q$ design that admits a tactical decomposition
   \[
     \V = \V_1 \sqcup \cdots \sqcup \V_m,\quad \B = \B_1 \sqcup \cdots \sqcup \B_n.
   \]
   Then $\Lambda_{lrs}=\Lambda_{lsr}$ and $|\V_{l}|\cdot\Lambda_{lrs}=|\V_{r}|\cdot\Lambda_{rls}$.
\end{theorem}
\begin{proof}
  It follows directly from Definition \ref{bigM} that $\Lambda_{rs}(P)=\Lambda_{sr}(P)$ for any point $P\in\V_{l}$, and hence that $\Lambda_{lrs}=\Lambda_{lsr}$.
  \par The equality $|\V_{l}|\Lambda_{lrs}=|\V_{r}|\Lambda_{rls}$ surely holds if $l=r$, so we can assume $l\neq r$. Let $\R = [\rho_{ij}]$ and $\K = [\kappa_{ij}]$ be the tactical decompostion matrices of the given tactical decompostion. Using Lemma \ref{firstequations} we find that
  \[
    \sum_{j=1}^{n} \rho_{lj}\kappa_{rj}\kappa_{sj}= \sum_{j=1}^{n} \frac{|\B_{j}|}{|\V_{l}|}\kappa_{lj}\kappa_{rj}\kappa_{sj}= \sum_{j=1}^{n} \frac{|\V_{r}|}{|\V_{l}|}\rho_{rj}\kappa_{lj}\kappa_{sj},
  \]
  hence
  \[
    |\V_{l}|\sum_{j=1}^{n} \rho_{lj}\kappa_{rj}\kappa_{sj}= |\V_{r}|\sum_{j=1}^{n}\rho_{rj}\kappa_{lj}\kappa_{sj}.
  \]
  Applying Theorem \ref{maintheorem} we find
  \begin{align*}
    &|\V_{l}|\left(\Lambda_{lrs}(\lambda_{2}-\lambda_{3})+\lambda_{3}|\V_{r}|\cdot|\V_{s}|\right)=|\V_{r}|\left(\Lambda_{rls}(\lambda_{2}-\lambda_{3})+\lambda_{3}|\V_{l}|\cdot|\V_{s}|\right)\\
    &\Leftrightarrow\quad|\V_{l}|\Lambda_{lrs}(\lambda_{2}-\lambda_{3})=|\V_{r}|\Lambda_{rls}(\lambda_{2}-\lambda_{3}),
  \end{align*}
  whence the equality $|\V_{l}|\Lambda_{lrs}=|\V_{r}|\Lambda_{rls}$ since $\lambda_{2}-\lambda_{3}\neq0$.
\end{proof}

\begin{theorem}
   Let $(\V, \B)$ be a $3$-$(v,k,\lambda_{3})_q$ design that admits a tactical decomposition
   \[
     \V = \V_1 \sqcup \cdots \sqcup \V_m,\quad \B = \B_1 \sqcup \cdots \sqcup \B_n.
   \]
   Then
   \[
     \sum^{m}_{s=1}\Lambda_{lrs}=\begin{cases}
                                   |\V_{l}|(q+1)+\frac{q^{v}-q^{2}}{q-1}-1 & l=r\\
                                   |\V_{r}|(q+1) & l\neq r.
                                 \end{cases}
   \]
\end{theorem}
\begin{proof}
  Let $P$ be a point of $\V_{l}$. We count the set
  \[
     \{(R,S)\in\V_{r}\times\V\, : \, \dim\langle P, R, S \rangle = 2\}
  \]
  in two ways. On the one hand,
  \begin{align*}
     &\{(R,S)\in\V_{r}\times\V\, : \, \dim\langle P, R, S \rangle = 2\}\\&=\bigsqcup^{m}_{s=1}\{(R,S)\in\V_{r}\times\V_{s}\, : \, \dim\langle P, R, S \rangle = 2\},
  \end{align*}
  so the size of this set equals $\sum^{m}_{s=1}\Lambda_{lrs}$ by Definition \ref{bigM} and Corollary \ref{independent}.
  \par On the other hand, if $l\neq r$, then for any point $R\in\V_{r}$ we find $q+1$ different $1$-spaces in the 2-dimensional space $\left\langle P,R\right\rangle$, so $q+1$ choices for the point $S$, therefore the size of this set equals
  \[
    |\V_{r}|(q+1).
  \]
  If $l=r$, then for any point in $R\in\V_{r}\setminus\{P\}$ we find $q+1$ different $1$-spaces in the $2$-dimensional space $\left\langle P,R\right\rangle$, so $q+1$ choices for the point $S$. For the point $R=P$, any point $S\in\V\setminus\{P\}$ determines a 2-dimensional space $\left\langle P,R,S\right\rangle=\left\langle P,S\right\rangle$. Hence, the size of this set for $l = r$ equals
  \[
    \left(|\V_{l}|-1\right)(q+1)+\left(\frac{q^{v}-1}{q-1}-1\right)\, = \, |\V_{l}|(q+1)+\frac{q^{v}-q^{2}}{q-1}-1.
  \]
  This concludes the proof.
\end{proof}

Recall that the 1-dimensional subspaces of $\V$ are called points. From now on, we call the 2-dimensional subspaces of $\V$ \emph{lines}. The set of lines of $\V$ will often be denoted by by $\Ll$. Note that the pair $(\V, \Ll)$ is a trivial $2$-$(v, 2, 1)_q$ design. Every group $G \le \PGammaL(\V)$ induces a tactical decomposition on $(\V, \Ll)$. This obtained tactical decomposition is closely related to parameter $\Lambda_{rls}$.

\begin{theorem}\label{thm:designlines}
  Let $(\V, \B)$ be a $3$-$(v,k,\lambda_{3})_q$ design that admits a $G$-induced tactical decomposition
  \[
    \V = \V_1 \sqcup \cdots \sqcup \V_m,\quad \B = \B_1 \sqcup \cdots \sqcup \B_n.
  \]
  Let $\Ll$ be the set of lines of $\V$. We consider the $G$-induced tactical decomposition of the trivial design $(\V, \Ll)$:
  \[
    \V = \V_1 \sqcup \cdots \sqcup \V_m, \qquad \Ll = \Ll_1 \sqcup \cdots \sqcup \Ll_{\omega},
  \]
  with associated tactical decomposition matrices $[\rho_{ij}^{\Ll}]$ and $[\kappa_{ij}^{\Ll}]$. Then
  \[
    \Lambda_{rls} =
    \begin{cases}
     \displaystyle{ \sum_{j = 1}^{\omega} \rho_{lj}^{\Ll} \kappa_{rj}^{\Ll} \kappa_{sj}^{\Ll} - \gs{v-1}{1}{q}} & l = r = s \\
      \displaystyle{\sum_{j = 1}^{\omega} \rho_{lj}^{\Ll} \kappa_{rj}^{\Ll} \kappa_{sj}^{\Ll} } &  \mbox{otherwise.}
    \end{cases}
  \]
\end{theorem}
\begin{proof}
  Let $P$ be a point of $\V_l$. From the definition of the coefficients $\rho_{ij}^{\Ll}$ and $\kappa_{ij}^{\Ll}$ it follows that each point of $\V_i$ lies on $\rho_{ij}^{\Ll}$ lines of $\Ll_j$, and each line of $\Ll_j$ contains $\kappa_{ij}^{\Ll}$ points of $\V_i$. Since $(\V, \Ll)$ is a $q$-Steiner system with $t=2$ we know that
  \begin{align*}
    \Lambda_{lrs}  &=  \left|\{(R, S) \in \V_r \times \V_s\, : \, \dim \langle P, R, S \rangle = 2 \}\right| \\
    &= \left| \{(R, S,\ell) \in \V_r \times \V_s\times\Ll \, : \, \langle P, R, S \rangle = \ell \}\right|.
  \end{align*}
  Counting the second set yields
	 \[
	  \Lambda_{rls} =
	  \begin{cases}
	    \displaystyle{ \sum_{j = 1}^{\omega} \rho_{lj}^{\Ll} \kappa_{rj}^{\Ll} \kappa_{sj}^{\Ll} - \gs{v-1}{1}{q}} & l = r = s \\
	    \displaystyle{ \sum_{j = 1}^{\omega} \rho_{lj}^{\Ll} \kappa_{rj}^{\Ll} \kappa_{sj}^{\Ll} }  &  \mbox{otherwise.}
	  \end{cases}\qedhere
	\]
\end{proof}

\begin{remark}\label{transitive}
  Recall that only a few examples of non-trivial $3$-designs over finite fields are known. These non-trivial examples were obtained using an automorphism group $G$ acting transitively on the set of points of the a $3$-$(v,k,\lambda_{3})_q$ design $(\V, \B)$. In this case (if the group is acting transitively on the set of points of $(\V, \B)$) there only one $\Lambda$ parameter, namely $\Lambda_{111}$. As a corollary of Lemma \ref{thm:designlines} or directly using Definition \ref{bigM} one can prove that
  \[
    \Lambda_{111} = q(q+2)\gs{v-1}{1}{q}.
  \]
  Unfortunately, if $G$ acts transitively on the points of $\V$ the results of Lemma \ref{secondequations} and Theorem \ref{maintheorem} cannot be used. For in this case, by Lemma \ref{firstequations} we know that $\kappa_{1j}=\gs{k}{1}{q}$ and hence the results of Lemma \ref{secondequations} and Theorem \ref{maintheorem} reduce to $\lambda_{1}=\sum^{n}_{j=1}\rho_{1j}$ by \eqref{tdesignissdesign}, a result we already know from Lemma \ref{firstequations}.
\end{remark}

The next theorem and the subsequent remark deal with $G$-induced tactical decompositions, with $|G|=p$ prime, having a fixed point. Note that such a fixed point (orbit of size one) is guaranteed to exist if $p$ is not a divisor of the number of points.

\begin{theorem}\label{primefixedpoint}
  Let $(\V, \B)$ be a $3$-$(v,k,\lambda_{3})_q$ design and let
   \[
     \V = \V_1 \sqcup \cdots \sqcup \V_m,\quad \B = \B_1 \sqcup \cdots \sqcup \B_n,
   \]
   be a $G$-induced tactical decomposition, $G\leq\Aut(\V, \B)$, with $|G|=p$ prime. If $\V_{l}$ is an orbit of size one, then $\Lambda_{lrs}\in\{0,1,p,p^{2}\}$ for $r,s=1,\dots,m$.
\end{theorem}
\begin{proof}
  Since $|G|$ is prime the orbits $\V_{i}$ have size $1$ or $p$. Let $P$ be the unique point in $\V_{l}$. A line $\ell$ through $P$ is either fixed by $G$ or else $\ell^{G}$ is an orbit of $p$ different lines through $P$. An orbit $\V_{i}$ of size $1$, $i\neq l$, is necessarily contained in a line through $P$ that is fixed by $G$; an orbit $\V_{i}$ of size $p$ is contained in a line through $P$ that is fixed or has one point in common with each line of an orbit $\ell^{G}$ of $p$ lines.
  \par If $l=r=s$, then $\Lambda_{lrs}=0$. If $l=r\neq s$, then $\Lambda_{lrs}=|\V_{s}|$ (analogously if $l=s\neq r$). Now, we assume that $r\neq l\neq s$. If $\V_{r}$ is an orbit of size one, equal to $\{R\}$, then the line $\left\langle P,R\right\rangle$ is fixed, and $\Lambda_{lrs}$ equals $0$ or $|\V_{s}|$. If $\V_{s}$ is an orbit of size one, the situation is equivalent. If both $\V_{r}$ and $\V_{s}$ are orbits of size $p$, then $\Lambda_{lrs}$ equals $p^{2}$ (in case $\V_{r}$ and $\V_{s}$ are on the same fixed line), $p$ (in case $\V_{r}$ and $\V_{s}$ have one point in common with each line of the same line orbit $\ell^{G}$) or $0$ (else).
\end{proof}

\begin{remark}\label{remarkprimefixedpoint}
  Given a $G$-induced tactical decomposition of a $q$-design $(\V, \B)$ with $t=3$ and $|G|=p$ prime, we want to compute the $\Lambda$-values related to at least one orbit of size $1$. Let $\V=\V_1 \sqcup \cdots \sqcup \V_m$ be the point set decomposition and let $\V_{l}=\{P\}$ be an orbit of size $1$. A line $\ell$ through $P$ is either fixed by $G$ or is contained in a line orbit of size $p$. If $\ell$ is fixed, it contains only entire orbits of points. If $\ell$ is not fixed, then no two points on $\ell$ belong to the same point orbit. Hence, we can define a partition $\Omega$ of the set $\{\V_{1}, \dots, \V_{l-1}, \V_{l+1}, \dots, \V_{m}\}$ in the following way: two orbits are in the same partition class if and only if they are on the same fixed line through $P$ or if they have one point in common with each line of the same line orbit of size $p$ through $P$. We find
  \[
    \Omega=\Omega_{1}\sqcup\dots\sqcup\Omega_{a}\sqcup\Omega_{a+1}\sqcup\dots\sqcup\Omega_{b},
  \]
  with each $\Omega_{i}=\{\V_{j_{1}},\dots,\V_{j_{k_{i}}}\}$. Let $\Omega_{1},\dots,\Omega_{a}$ be the partition classes that correspond to fixed lines and let $\Omega_{a+1},\dots,\Omega_{b}$ be the partition classes that correspond to line orbits through $P$ of size $p$.
  \par Let $r\neq l\neq s$. Following the arguments in the proof of Theorem \ref{primefixedpoint} the value $\Lambda_{lrs}$ equals $0$ if $\V_{r}$ and $\V_{s}$ do not belong to the same partition class $\Omega_{i}$. If they do belong to the same partition class $\Omega_{i}$, then $\Lambda_{lrs}=|\V_{r}|\cdot|\V_{s}|$ if $1\leq i\leq a$ and $\Lambda_{lrs}=|\V_{r}|=|\V_{s}|=p$ if $a<i\leq b$. To summarize,
  \[
	  \Lambda_{lrs} =
	  \begin{cases}
		  |\V_{r}|\cdot|\V_{s}| & \V_r, \V_s \in \Omega_i, \, i \leq a, \\
		  p & \V_r, \V_s \in \Omega_i, \, a < i \leq b,\\
		  0 & \mbox{otherwise}. \\
	  \end{cases}
  \]
\end{remark}

The final theorem in this section gives a bound on the value of $\Lambda_{lrs}$ given a tactical decomposition induced by a group of prime order. Considering Remark \ref{remarkprimefixedpoint} this is of interest when the three orbits involved are not fixed points.

\begin{theorem}
   Let $(\V, \B)$ be a $3$-$(v,k,\lambda_{3})_q$ design and let
   \[
   \V = \V_1 \sqcup \cdots \sqcup \V_m,\quad \B = \B_1 \sqcup \cdots \sqcup \B_n,
   \]
   be a $G$-induced tactical decomposition, $G\leq\Aut(\V, \B)$, with $|G|=p$ prime. If $\Lambda_{lrs}\neq p^{2}$, then $\Lambda_{lrs}\leq p\sqrt{p-\frac{3}{4}}+\frac{p}{2}$, for $l,r,s=1,\dots,m$.
\end{theorem}
\begin{proof}
  Let $\V_{l}$, $\V_{r}$ and $\V_{s}$ be three point orbits. If $1\in\{|\V_{l}|,|\V_{r}|,|\V_{s}|\}$, the result follows from Theorem \ref{primefixedpoint}. So, we assume $|\V_{l}|=|\V_{r}|=|\V_{s}|=p$. Let $P$ be a point of $\V_{l}$.
  \par We first show that an orbit $\V_{j}$, $j\neq l$ and with $|\V_{j}|=p$, has either $p$ or else at most $K_{p}=\sqrt{p-\frac{3}{4}}+\frac{1}{2}$ points in common with a line through $P$. Let $\ell$ be a line through $P$ having at least one point $Q$ in common with $\V_{j}$. If $n_{\ell}=|\ell\cap\V_{j}|<p$, then $\V_{j}\not\subseteq\ell$, hence the line $\ell$ is not fixed, so $\ell^{G}$ is a set of $p$ different lines. Now, let $\{g_{1},\dots,g_{n_{\ell}}\}$ be the elements of $G$ that map the points of $\ell\cap\V_{j}$ onto $Q$. The lines $\ell^{g_{1}},\dots,\ell^{g_{n_{\ell}}}$ are $n_{\ell}$ different lines through $Q$, each containing $n_{\ell}$ distinct points of $\V_{j}$ (including $Q$). Hence, we find at least $n_{\ell}(n_{\ell}-1)+1$ points of $\V_{j}$. As $|\V_{j}|=p$, necessarily $n^{2}_{\ell}-n_{\ell}+1\leq p$. It follows that $n_{\ell}\leq\sqrt{p-\frac{3}{4}}+\frac{1}{2}$.
  \par Now let $\ell_{1},\dots,\ell_{c}$ be the lines through $P$ containing at least one point of $\V_{r}$ and at least one point of $\V_{s}$. Denote $|\ell_{i}\cap\V_{r}|$ and $|\ell_{i}\cap\V_{s}|$ by $n_{i}$ and $n'_{i}$ respectively, $i=1,\dots,c$. Then $\Lambda_{lrs}=\sum^{c}_{i=1}n_{i}n'_{i}$. Moreover, $\sum^{c}_{i=1}n_{i}\leq p$ and $\sum^{c}_{i=1}n'_{i}\leq p$. If there is a line through $P$ containing $p$ points of $\V_{r}$, then this line is fixed (necessarily $c\in\{0,1\}$). If $\V_{s}$ is also contained in this line, then $\Lambda_{lrs}=p^{2}$ and if it is not, then $\Lambda_{lrs}=0$. So, now we can assume that $1\leq n_{i},n'_{i}\leq K_{p}$. By the Cauchy-Schwarz inequality
  \[
    \sum^{c}_{i=1}n_{i}n'_{i}\leq\sqrt{\left(\sum^{c}_{i=1}n^{2}_{i}\right)\left(\sum^{c}_{i=1}n'^{2}_{i}\right)}.
  \]
  Since $1\leq n_{i},n'_{i}\leq K_{p}$, it is immediate that
  \[
    \sum^{c}_{i=1}n^{2}_{i}\leq \frac{p}{K_{p}}K^{2}_{p}=pK_{p} \text{ and analogously } \sum^{c}_{i=1}n'^{2}_{i}\leq pK_{p}.
  \]
  So,
  \[
    \sum^{c}_{i=1}n_{i}n'_{i}\leq\sqrt{(pK_{p})(pK_{p})}=pK_{p},
  \]
  which proves the inequality.
\end{proof}

\section{Application to known designs}\label{examples}

In this final section we will discuss the application of the results in Section \ref{sec:3-designs} to some known $3$-designs. First we look at a design $(\V,\B)$ with parameters $3$-$(4,3,1)_{2}$, where $\V=\F^{4}_{2}$. One can see directly that this is the design of all $3$-spaces in $\V$, but as proof of concept we construct it using a tactical decomposition based on a prescribed automorphism group. We consider the group $G\leq\PGammaL(\V)$ generated by
\[
  \begin{bmatrix}
    0&0&0&1\\
    0&0&1&0\\
    0&1&1&0\\
    1&0&0&1
  \end{bmatrix}.
\]
The group $G$ is a cyclic group of order $3$; its action on the points of $\V$ yields five orbits:
\[
  \V=\V_{1}\sqcup \cdots \sqcup \V_{5},
\]
with orbit representatives $\left\langle\left[1,0,0,0\right]\right\rangle$, $\left\langle\left[1,0,1,0\right]\right\rangle$, $\left\langle\left[1,0,1,1\right]\right\rangle$, $\left\langle\left[1,1,0,0\right]\right\rangle$ and $\left\langle\left[0,1,0,0\right]\right\rangle$. Each of these five orbits has size $3$. Now we assume that $G$ is an automorphism group of the design $(\V, \B)$. We do not know the orbits of $\B$ under the action of $G$ as we do not know which $3$-spaces of $\V$ are blocks of $(\V, \B)$. However, all orbits of $\B$ must have size $3$ since no $3$-space of $\V$ can be fixed by $G$. Indeed, a fixed 3-space contains only full point orbits. Since each point orbit is of size 3, and a 3-space contains $7$ points, G cannot fix any 3-space of $\V$. So, and since $|\B|=15$ we can write
\[
  \B=\B_{1}\sqcup \cdots \sqcup \B_{5},
\]
with $\B_{1},\dots,\B_{5}$ the orbits of $\B$ under the action of $G$.
\par We consider the corresponding tactical decomposition matrices $[\rho_{ij}]$ and $[\kappa_{ij}]$. Note that by Lemma \ref{firstequations} we have $\rho_{ij}=\kappa_{ij}$ for all $i,j$, since all point orbits and all block orbits have size 3. First we need to calculate the values $\Lambda_{lrs}$ for $l,r,s\in\{1,\dots,5\}$. An easy calculation yields
\[
  \Lambda_{lrs}=\begin{cases}
                  8&l=r=s\\
                  1&l\neq r\neq s\neq l\\
                  3&\text{else.}
                \end{cases}
\]
Now we apply Lemma \ref{firstequations}, Theorem \ref{secondequations} and Theorem \ref{maintheorem}. Note that $\lambda_{3}=1$, $\lambda_{2}=3$ and $\lambda_{1}=7$. We find
\begin{align}\label{ex1.1}
	\sum^{5}_{j=1}\rho_{ij}&=7,\qquad \sum^{5}_{i=1}\rho_{ij}=\gs{3}{1}{2}=7,\\\label{ex1.2}
	\sum_{j = 1}^{5} \rho_{lj} \rho_{rj} &= 
	\begin{cases}
	  13 & l = r \\
	  9 & l \neq r,
	\end{cases}\\\label{ex1.3}
	\sum_{j=1}^{5} \rho_{lj}\rho_{rj}\rho_{sj} &=
	\begin{cases}
	  31 & l=r=s\\
	  11 &l\neq r\neq s\neq l\\
          15 &\mbox{else.}
	\end{cases}
\end{align}
By looking at the relations (\ref{ex1.1}), (\ref{ex1.2}), (\ref{ex1.3}) for a row (row sum, row sum of squares, row sum of third powers) it follows directly that a row of the matrix $[\rho_{ij}]$ must be a permutation of
$$
\{3, 1, 1, 1, 1\}.
$$
By using relation (\ref{ex1.2}) it can be seen that no two $3$'s can be in the same column. Hence, the matrix $[\rho_{ij}]$ should equal
\[
  \begin{bmatrix}
    3&1&1&1&1\\
    1&3&1&1&1\\
    1&1&3&1&1\\
    1&1&1&3&1\\
    1&1&1&1&3
  \end{bmatrix}
\]
up to a rearrangement of rows and columns, and indeed this matrix satisfies all of the above relations. We find the design $(\V,\B)$ consisting of all $3$-spaces in $\V=\F^{4}_{2}$.

Now we look at the $3-(8,4,\lambda)_{2}$ designs $(\V, \B)$ that were studied in \cite{BKL07}. Here $\V=\F^{8}_{2}$ and $\B$ is a set of $4$-spaces. These $q$-designs admit (by construction) the normaliser of a Singer cycle as an automorphism group; this automorphism group $G$ is generated by the elements
\[
  \begin{bmatrix}
    1&0&0&0&1&0&1&1\\
    0&0&0&0&0&0&0&1\\
    0&1&0&0&1&1&1&0\\
    0&0&0&0&1&0&1&0\\
    0&0&1&0&1&1&0&1\\
    0&0&0&0&0&1&0&0\\
    0&0&0&1&0&1&1&0\\
    0&0&0&0&0&0&1&0
  \end{bmatrix}
  \text{ and }
  \begin{bmatrix}
    0&0&0&0&0&0&0&1\\
    1&0&0&0&0&0&0&0\\
    0&1&0&0&0&0&0&1\\
    0&0&1&0&0&0&0&1\\
    0&0&0&1&0&0&0&1\\
    0&0&0&0&1&0&0&0\\
    0&0&0&0&0&1&0&0\\
    0&0&0&0&0&0&1&0
  \end{bmatrix}\;.
\]
The group $G$ acts transitively on the point set of $\V$. In this case the comments of Remark \ref{transitive} apply and we need only to consider the result from Lemma \ref{firstequations}. The group $G$ has 109 orbits on the 4-spaces of $\F^{8}_{2}$ and we know that a $3-(8,4,\lambda)_{2}$ design which admits $G$ as its automorphism group, is the union of some of these orbits, $\B_{1},\dots,\B_{n}$. Since $|\V_{1}|=|\V|=255$ and $\kappa_{1j}=\gs{4}{1}{2}=15$ we know that $\rho_{1j}=\frac{|\B_{j}|}{17}$.
\par Considering the orbits of the $4$-spaces under the action of $G$, we find that there are 92 orbits having size $120\cdot17$, ten orbits having size $60\cdot17$, five orbits having size $30\cdot17$, one orbit having size $20\cdot17$ and one orbit having size $17$. By \eqref{tdesignissdesign} and Lemma \ref{firstequations} we obtain that
\[
  \sum^{n}_{j=1}\rho_{1j}=\lambda_{1}=127\cdot3\cdot\lambda\equiv\lambda\pmod{10}\;.
\]
Except possibly one, all values $\rho_{ij}$ equal $0$ modulo $10$. Hence, $\lambda$ must equal $1$, $10$, $11$, $20$, $21$, $30$ or $31$. Note that $\lambda=31$ corresponds to a trivial design and that $20$, $21$ and $30$ can only occur as complements of $11$, $10$ and $1$ respectively. In \cite{BKL07} all values from 1 up to 30 were tested for $\lambda$, but using tactical decomposition arguments that search could have been restricted to three values for $\lambda$. This shows the usefulness of this technique. 

\vspace*{0.5cm}
\par \textbf{Acknowledgment:} The research of the first author is supported by the BOF-UGent (Special Research Fund of Ghent University). The research of the second author has been supported by the Croatian Science Foundation under the project 7766.


\begin{thebibliography}{99}
\bibitem{ACLY00}
 R. Ahlswede, N. Cai, S.Y.R. Li and R.W. Yeung, Network information flow, IEEE Transactions on Information Theory 46(4) (2000), 1204--1216.
  \bibitem{B13} M. Braun, New 3-designs over the binary field, Int. Electron. J. Geom. 6 (2013), no. 2, 79--87. 
  \bibitem{BEOVW15} M. Braun, T. Etzion, P. R. J. Ostergaard, A. Vardy and A. Wassermann, On the existence of $q$-analogs of Steiner systems, to appear in the Forum of Mathematics, PI. (2015)
  \bibitem{BKL07} M. Braun, A. Kerber and R. Laue, Systematic construction of $q$-analogs of designs, Des. Codes Cryptogr. 34 (2005), no. 1, 55--70.
  \bibitem{BKN15} M. Braun, M. Kiermaier and A. Naki\'c, On the automorphism group of a binary $q$-analog of the Fano plane, European Journal of Combinatorics 51 (2016), 443--457.
  \bibitem{CI74} P. Cameron: Generalisation of Fisher's inequality to fields with more than one element, In: McDonough T.P., Mavron V.C. (eds.) Combinatorics: Proceedings of the British Combinatorial Conference 1973, London Mathematical Society Lecture Notes Series, vol. 13, 9--13, Cambridge University Press, 1974. 
  \bibitem{CII74} P. Cameron, Locally symmetric designs, Geom. Dedicata 3 (1974), 56--76.
  \bibitem{D76} P. Delsarte, Association schemes and $t$-designs in regular semilattices, J. Combin. Theory Ser. A 20 (1976), no. 2, 230--243.
  \bibitem{D68} P. Dembowski, Finite geometries, Springer, Berlin/Heidelberg/New York, 1968.
  \bibitem{E13} T. Etzion, Problems on $q$-Analogs in Coding Theory, arXiv:1305.6126 (2013).
  \bibitem{EV11} T. Etzion and  A. Vardy, On $q$-analogs of Steiner systems and covering designs. Adv. Math. Commun. 5 (2011), no. 2, 161--176.
  \bibitem{GAP}
  The GAP Group, GAP – Groups, Algorithms, and Programming, Version
  4.4.12, 2008, (http://www.gap-system.org).
  \bibitem{HMK}
  T. Ho, M. M\'edard, R. K\"otter, D.R. Karger, M. Effros, J. Shi and B. Leong, A
  random linear network coding approach to multicast, IEEE Transactions on Information
  	Theory 52 (2006), no. 10, 4413--4430.
  \bibitem{JT85} Z. Janko and T. van Trung, Construction of a new symmetric block design for $(78,22,6)$ with the help of tactical decompositions, J. Combin. Theory A 40 (1985), 451--455.
  \bibitem{KP14} M. Kiermaier and M. O. Pav\v cevi\'c, Intersection numbers for subspace designs, J. Combin. Des. 23 (2015), no. 11, 463--480.
  \bibitem{KK08} R. Koetter and F. R. Kschischang, Coding for errors and erasures in random network coding. IEEE Trans. Inform. Theory 54 (2008), 3579--3591.
  \bibitem{KNP11} V. Kr\v{c}adinac, A. Naki\'{c} and M.O. Pav\v{c}evi\'{c}, The Kramer-Mesner method with tactical decompositions: some new unitals on $65$ points, J. Combin. Des. 19 (2011), no. 4, 290--303.
  \bibitem{KNP13} V. Kr\v{c}adinac, A. Naki\'{c} and M.O. Pav\v{c}evi\'{c}, Equations for coefficients of tactical decomposition matrices for $t$-designs, Des. Codes Cryptogr. 72 (2014), no. 2, 465--469.
  \bibitem{L13}
  L. Lambert, Random network coding and designs over $\Fq$, Master thesis, Ghent University, 2013.
  \bibitem{nak1} A. Naki\'{c} and M.O. Pav\v{c}evi\'{c}, Tactical decompositions of designs over finite fields, Des. Codes Cryptogr. 77 (2015), no. 1, 49--60 .
  \bibitem{MR07} R. Mathon and  A. Rosa, $2$-$(v,k,\lambda)$ designs of small order, In: Colbourn C.J., Dinitz J.H. (eds.) The Handbook of Combinatorial Designs, Second Edition, CRC Press, 2007.
  \bibitem{M99} K. Metsch, Bose-Burton type theorems for finite projective, affine and polar spaces. In Lamb and Preece, editors, Surveys in Combinatorics, Lecture   Notes Series, volume 267. London Mathematical Society, 1999.
  \bibitem{MMY95} M. Miyakawa, A. Munemasa and S. Yoshiara, On a class of small 2-designs over $GF(q)$. J. Combin. Des. 3 (1995), 61--77.
  \bibitem{S90} H. Suzuki, 2-Designs over $GF(2m)$, Graphs Combin. 6 (1990), 293--296.
  \bibitem{S92} H. Suzuki, 2-Designs over $GF(q)$, Graphs Combin. 8 (1992), 381--389.
  \bibitem{ST87} S. Thomas, Designs over finite fields, Geom. Dedicata 24 (1987), no. 2, 237--242.
  \bibitem{YLCZ06}
  R. W. Yeung, S.-Y. R. Li, N. Cai and Z. Zhang, Network Coding Theory. Boston, MA, Now, 2006.
  
\end{thebibliography}
\end{document}